\documentclass{amsart}
\usepackage{amssymb,tikz-qtree}

\newcommand\Bun{\operatorname{\mathfrak{B}un}}

\renewcommand{\hat}[1]{\widehat{#1}}

\newcommand\dl{\mathrm{dl}} 
\newcommand\iR{\mathrm{\cR}} 
\newcommand\rR{\mathrm{\cR_{\bbR}}}

\usepackage[all]{xy}
\usepackage{eurosym, mathrsfs}

\newcommand\boxb[1]{\square_b}

\newcommand\paperintro%
        {%
         }
\newcommand\paperbody%
        {%
         }

\newtheorem{theorem}{Theorem}

\newtheorem{corollary}[theorem]{Corollary}

\newtheorem{proposition}[theorem]{Proposition}

\numberwithin{equation}{section}

\theoremstyle{remark}
\newtheorem{definition}{Definition}

\newcommand\cFTs{{}^{\Phi}\overline{T}\kern-1pt{}^*}

\newcommand\Tr{\operatorname{Tr}}

\newcommand\Ch{\operatorname{Ch}}

\newcommand\tphi{\tilde{\phi}}

\newcommand\cI{\mathcal{I}}

\newcommand\cR{\mathcal{R}}

\newcommand\hN{\widehat{N}}
\newcommand\vN{\check{N}}
\newcommand\vQ{\check{Q}}

\newcommand\hGa{\widehat{\Gamma}}

\newcommand\hQ{\widehat{Q}}

\newcommand\bbC{\mathbb C}

\newcommand\bbR{\mathbb R}
\newcommand\bbS{\mathbb S}

\newcommand\bbZ{\mathbb Z}

\newcommand\CI{{\mathcal{C}}^{\infty}}

\newcommand\cFNs{{}^{\Phi}\overline N\kern-1pt{}^*}

\newcommand\red{\operatorname{red}}

\newcommand\tr{\operatorname{tr}}

\newcommand\ev{\operatorname{even}}

\newcommand\Id{\operatorname{Id}}

\newcommand\UU{\operatorname{U}}

\newcommand\loc{\operatorname{loc}}

\newcommand\Mand{\text{ and }}

\newcommand\Miff{\text{ iff }}

\newcommand\Mthen{\text{ then }}

\begin{document}
\title[K-theory and Resolution, II]
{Equivariant K-theory and Resolution\\
II: Non-Abelian actions}

\author{Panagiotis Dimakis and Richard Melrose}
\address{Department of Mathematics, Stanford University}
\address{Department of Mathematics, Massachusetts Institute of Technology}
\email{pdimakis@stanford.edu, rbm@math.mit.edu}

\begin{abstract}

The smooth action of a compact Lie group on a compact manifold can be
resolved to an iterated space, as made explicit by Pierre Albin and the
second author. On the resolution the lifted action has fixed isotropy type
corresponding to the open stratum and also in an iterated sense, with
connecting equivariant fibrations over the boundary hypersurfaces covering
the resolutions of the other strata. This structure descends to a
resolution of the quotient as a stratified space. For an Abelian group
action the equivariant K-theory can then be described in terms of bundles
over the bases `dressed' by the representations of the isotropy types with
morphisms covering the connecting maps. A similar model is given here
covering the non-Abelian case. Now the reduced objects are torsion-twisted
bundles over finite covers of the bases, corresponding to the projective
action of the normalizers on the representations of the isotropy groups,
again with morphisms over all the boundaries. This leads to a closely
related iterated deRham model for equivariant cohomology and, now with
values in forms twisted by flat bundles of representation rings over the
bases, for delocalized equivariant cohomology. We show, as envisioned by
Baum, Brylinksi and MacPherson, that the usual equivariant Chern character,
mapping to equivariant cohomology, factors through a natural Chern
character from equivariant K-theory to delocalized equivariant cohomology
with the latter giving an Atiyah-Hirzebruch isomorphism.
\end{abstract}

\maketitle
\tableofcontents

\paperintro
\section*{Introduction}

In this continuation of \cite{EqKtRe} the description, in terms of the
normal resolution, of the equivariant K-theory of a compact manifold with an
action by a compact Lie group is extended to the general, non-Abelian,
case. Doing so suggests a deRham realization of
\emph{delocalized} equivariant cohomology, in the sense of Baum, Brylinksi
and MacPherson \cite{MR86g:55006}. The natural Chern character
\begin{equation}
\Ch:K^0_G(M)\longrightarrow H^{\ev}_{G,\dl}(M)
\label{EKR.98}\end{equation}
yields an Atiyah-Hirzebruch isomorphism over $\bbR.$

The simplest cases of smooth $G$ action can be analyzed without
resolution. If $G$ acts freely then the compact manifold is the total space
of a principal bundle $\kappa :M\longrightarrow Y,$ over the smooth orbit
space $Y=G\backslash M.$ The action on an equivariant bundle, $V,$ over $M$
gives descent data defining a bundle $W$ over $Y.$ Conversely, every
equivariant bundle is equivariantly isomorphic to the pull-back, $V=\kappa
^*W,$ of a bundle over $Y$ with $G$ action $gw_x=w_{gx}$ and equivariant
bundle isomorphisms over $M$ descend uniquely to bundle isomorphisms over
$Y.$ Passing to the Grothendieck groups this induces a natural
identification as rings
\begin{equation}
K_G(M)=K(Y).
\label{E-K-n.256}\end{equation}
Thus the equivariant K-theory of $M,$ as a ring, does \emph{not} depend on the bundle
structure but only on the quotient. Note however that the representation
ring of $G$ acts on $K_G(X)$ by taking the tensor product with a
representation viewed as an equivariant bundle; the composite is a
`transfer map' 
\begin{equation}
\tau:\iR(G)\longrightarrow K(Y)
\label{EKR.111}\end{equation}
which does involve the geometry of the principal bundle.

The second elementary case is that of a trivial action. Since this
corresponds to the isotropy group below we denote the compact Lie group
acting as $H.$ Let $H'$ be a complete set of irreducible unitary
representations, with $M(\pi)$ the Hermitian vector space on which $\pi\in H'$
acts
\begin{equation}
H'\ni\pi\Longrightarrow \pi :H\longrightarrow \UU(M(\pi)).
\label{E-K-n.297}\end{equation}
Each fiber of an $H$-equivariant bundle over $M$ may then be decomposed
under the action to a sum
\begin{equation}
V_x=\bigoplus_{\pi\in H'} V_x(\pi),\
V(\pi)=W(\pi)\otimes M(\pi)
\label{E-K-n.298}\end{equation}
where the $W(\pi)$ form an unrestricted bundle with compact support on
$M\times H'.$

An $H$-equivariant isomorphism of $V$ induces a bundle isomorphism of each
of the $W(\pi)$ and the decomposition of representations under tensor product
then results in the identification as rings
\begin{equation}
K_H(M)=K(M)\otimes_{\bbZ}\cR(H).
\label{E-K-n.257}\end{equation}

The more general case in which the $G$-action has a single stratum is
intermediate between these two cases, where now all the isotropy groups are
conjugate. Borel showed that such a smooth action of a compact Lie group
with fixed isotropy type can be reduced to a principal action by the
quotient $Q=N(H)/H$ of the normalizer in $G$ for a choice of isotropy group
\begin{equation}
\xymatrix{
Q\ar@{-}[r]&X^H\ar@{^(->}[rr]\ar[dr]&&X\ar[dl]\\
&&Y=X/G
}
\label{EKR.101}\end{equation}
where $X^H$ is the submanifold where the isotropy group is $H.$ As rings
\begin{equation}
K_G(X)=K_N(X^H).
\label{EKR.163}\end{equation}

In this setting the equivariant K-theory was identified by Wassermann
\cite{Wassermann} in terms of twisted K-theory. Our description of K-theory in the
general case follows from a related decomposition of equivariant vector bundles in
terms of twisted bundles, or gerbe modules, rather than the $C^*$-theoretic
approach in \cite{Wassermann}.

Thus, suppose that $H\subset N$ is a closed normal subgroup (the notation
indicating that $N$ is perhaps the normalizer of $H$ in a larger group)
acting, with fixed isotropy group $H,$ on a compact manifold (allowed to
have corners) $Z;$ so $Z$ is a principal $Q=N/H$ bundle over a compact
manifold with corners $Y.$ The action of $N$ on $H$ by conjugation induces
a locally trivial action of the quotient group $Q=N/H$ on the irreducible
representations, $H',$ so with finite orbits. This action lifts to an
action on $Z\times H'$ with the same isotropy group, $H.$ Our `reduced'
description of $N$-equivariant bundles over $Z$ utilizes the finite covers
\begin{equation}
Y(\pi)=(Z\times N\pi)/N\longrightarrow Y
\label{EKR.108}\end{equation}
given by replacing $H'$ by the orbit $N\pi\subset H'.$

Schur's Lemma shows that the projective action of $N$ on the representation
space of $\pi$ lifts to a projective action of $N$ on the sum over the
orbit of the representations and so to a central extension
\begin{equation}
\xymatrix{
\UU(1)\ar[r]&\vN(\pi)\ar[r]&N(\pi)
}
\label{EKR.102}\end{equation}
which corresponds to Mackey's obstruction. Since $H$ lifts naturally into
$\vN(\pi)$ this defines a central extension $\vQ(\pi)$ of $Q$ for each
orbit. Making a choice of representatives $\pi\in H'$ parameterizing the
orbit space $H'_N=H'/N$ the space
\begin{equation}
\xymatrix{
Z(\pi)= (Z\times N\pi)\ar[r]^-{Q}&Y(\pi)
}
\label{EKR.103}\end{equation}
is a principal bundle for the action of $Q.$

The decomposition \eqref{E-K-n.298}  may be captured by considering the
$H$-equivariant homomorphisms
\begin{equation}
W_x(\pi)=\{l:M(\pi)\longrightarrow V_{x}(\pi); h\circ
l=l\circ\pi(h)\ \forall\ h\in H \}.
\label{E-K-n.300}\end{equation}
This has dimension, $n,$ the multiplicity of $\pi$ in $V_x.$ The rigidity
of representations of compact groups means that $n$ is constant on
components of $X.$ Let $W(\pi)$ denote the collective bundle over $Z(\pi)$
with fiber $W_x(\pi')$ at $(x,\pi')\in (X,N\pi).$

The action of $N$ on $V$ induces a $\hQ(\pi)$-equivariant action on
$W(\pi)$ by the reverse extension to $\vQ(\pi).$ This makes the $W(\pi)$
into twisted bundles over $Y(\pi)$ or equivalently into gerbe modules for
the lifting bundle gerbe defined by the central extension $\hQ(\pi)$ and
the principal bundle $Z(\pi)\longrightarrow Y(\pi).$ Functorially this is
associated with the (torsion) Dixmier-Douady class $\alpha (\pi)\in
H^3(Y(\pi);\bbZ)$ which is the transgression of the Mackey class. Denoting
the category of twisted bundles $\Bun(Y(\pi);\alpha (\pi))$ (in which the
morphisms depend on the trivialization of $\alpha (\pi))$ this gives an
equivalence of categories
\begin{equation}
\Bun_N(Z)\longrightarrow \bigoplus_{\pi\in H'_N}\Bun(Y(\pi);\alpha (\pi))
\label{EKR.104}\end{equation}
where the sums on the right are finite. Wassermann's
(\cite{Wassermann}) description of the equivariant K-theory for an action
with fixed isotropy type
\begin{equation}
K_G(M)\simeq\bigoplus_{\pi\in H'_N} K(Y(\pi);\alpha (\pi))
\label{EKR.105}\end{equation}
is a direct consequence. Note that this is a significantly functorial
statement because the twisted K-groups on the right depend on the
representations of the Dixmier-Douady class with different representations
leading to isomorphisms corresponding to tensoring with line bundles.

The localized representation ring, $L(H),$ consisting of the $H$-invariant
polynomials on the Lie algebra of $H$ (so $L(H)$ is the real representation
ring if $H$ is connected), is acted upon by the normalizer, $N,$ of $H$ and
hence by the quotient $Q.$ This action is locally trivial so the quotient
\begin{equation}
Z\times L(H)/N=L_H(Y)
\label{EKR.152}\end{equation}
is a flat bundle over $Y$ modeled on $L(H).$ By an argument similar to that
for free actions (see for example \cite{Guillemin-Sternberg},
\cite{Meinrenken-Enc}), the Cartan model for the equivariant cohomology of
$X$ can be reduced to the deRham cohomology of the complex of forms with
values in $L_H$
\begin{equation}
\CI(YZ;\Lambda ^*\otimes L_H).
\label{EKR.153}\end{equation}

Replacing $L(H)$ by the real representation ring generated by $H',$
$\rR(H)=\oplus(H'\otimes\bbR),$ the quotient of $Z\times\rR(H)$ by the
action of $Q$ is again the total space of a flat coefficient bundle
which we denote $\rR(Y).$ Then delocalized equivariant cohomology, in the
sense of Baum, Brylinski and MacPherson, corresponds to the deRham
cohomology of
\begin{equation}
\CI(Y;\Lambda ^{*}\otimes\rR(Y)).
\label{EKR.106}\end{equation}
As a consequence of the properties of torsion-twisted K-theory, in
particular that the Chern character is a deRham class in the usual sense,
the Chern character applied to the twisted bundles $W(\pi)$ over $Y(\pi)$
gives an element of \eqref{EKR.106} and defines the delocalized Chern
character \eqref{EKR.98}. The usual equivariant Chern character to equivariant deRham
cohomology corresponds to applying the twisted Chern character to the
(twisted) coefficient bundles defined by the elements of $H'$ and so
factors through \eqref{EKR.98} by a surjective localization map.

A general smooth action by a compact Lie group, $G,$ on a compact manifold,
$M,$ is reduced to an iterated version of the case of a single isotropy
type by resolution. The `normal resolution' of the action (which may not be
minimal as a resolution), see \cite{MR2560748}, is obtained by the radial
blow up of the isotropy types, the strata, of the action in any order
compatible with the (reversed) partial inclusion order on the isotropy
classes. The result is a compact manifold with corners with iterated
blow-down map
\begin{equation}
\beta :X\longrightarrow M
\label{EKR.99}\end{equation}
with interior the preimage of the open isotropy type. The resolutions of
the closures, $M_a,$ of the components of the non-open isotropy types are
in 1-1 correspondence with the boundary hypersurfaces $H_a$ of $X$ as the
images of equivariant fibrations 
\begin{equation}
\phi_a:H_a\longrightarrow X_a,\ \beta _a:X_a\longrightarrow M_a
\label{EKR.100}\end{equation}
with the action on each $X_a$ having the same conjugacy class of closed
subgroups as the action on the interior of $M_a.$ These fibrations give $X$
an iterated structure, in particular the fibrations are jointly compatible
at each boundary face of $X,$ forming a chain under fiber inclusion with
respect to the isotropy partial order.

As shown in \cite{EqKtRe} smooth equivariant bundles over $M$ lift to
smooth equivariant bundles over each of the $X_a,$ related by pull-back
under the $\phi_a.$ This functor to iterated equivariant bundles over $X$
projects to an equivalence on the Grothendieck groups, i.e.\ the iterated
equivariant bundles on $X$ lead to the equivariant K-theory of $M.$

On each of the compact manifolds with corners, $X_a,$ including $X=X_0,$
the description of equivariant bundles proceeds through the Borel reduction
to corresponding manifolds $Z_a$ where the action has a fixed isotropy
group. The quotient manifolds $Y_a$ with quotient fibrations $\phi_a$ form
an iterated manifold giving a resolution of the stratified space $M/G.$ The
main step in defining the iterated twisted bundles representing equivariant
K-theory on the resolution of the quotient is to describe the pull-back
maps corresponding to the equivariant fibrations $\phi_a.$ It is
significant here that the existence of an equivariant fibration for a
space, $H_a\longrightarrow X_a$ with $G$-actions on both $H_a$ and $X_a$
having single isotropy type, imposes a partial triviality condition on the
action on $H_a.$

Changing notation, let $\phi:X_0\longrightarrow X_1$ be a fibration of
smooth manifolds with corners, equivariant for $G$-actions with fixed
isotropy type, and with quotient fibration $\tphi:Y_0\longrightarrow Y_1.$
We may assume that $Y_0$ is connected. First pass to the Borel reduction
$Z_1=X^{H_1}_1$ of $X_1$ for a choice of isotropy group $H_1,$ with
normalizer $N_1.$ The preimage $Z_0=\phi^{-1}(Z_1)$ is a `reduction' of
$X_0$ generalizing Borel's reduction and made precise below. Choose a
component of $Z_0$ and an isotropy group $H_0$ at some point in it. The
isotropy groups at all points of this component are necessarily $H_1$
conjugate to $H_0,$ not just conjugate under the $N_1$-action. Let
$Z_0'\subset Z_0$ be the union of the components where all isotropy groups
are $H_1$ conjugate to $H_0,$ this is a reduction of $Z_0,$ and hence
$X_0,$ with group action by $N_1'\subset N_1$ the subgroup under which
$H_0$ is mapped to an $H_1$ conjugate. The image, $Z_1'=\phi(Z_0')$ is a
reduction of $Z_1$ with the same group, $N_1'$ acting. Passing to the Borel
reduction $Z_0''\subset Z_0'$ with its $N'_1$-action gives spaces with
isotropy groups $H_0\subset H_1$ and actions by groups $N_0$ and $N_1'$
connected by an equivariant fibration. As for Borel reduction under these
successive steps the $G$-equivariant bundles over $X_0$ and $X_1,$
respectively, are identified with with $N_0$ and $N_1'$-equivariant bundles
over $Z''_0$ and $Z_1'.$

This `double' reduction of the $G$-action on the boundary hypersurfaces
$H_a\subset X$ corresponds to a refinement of the decomposition of
$G$-equivariant bundles into twisted bundles and allows pull-back under
$\phi_1$ for equivariant bundles to be identified with an `augmented'
pull-back, $\phi_1^{\#},$ on the corresponding twisted bundles. Applied
iteratively this gives the compatibility conditions on twisted bundles and
a functorial reduction
\begin{equation}
\begin{gathered}
\Bun_G(X)\equiv \Bun_{\red}(Y_*)\subset\bigoplus_{\pi\in H'_N}\Bun(Y(\pi)
;\alpha (\pi)),\\
\oplus_{\alpha } W_\alpha \in \Bun_{\red}(Y_*)\Longleftrightarrow 
W_\alpha \big|_{H_\beta }=\phi^{\#}_{\alpha ,\beta } W_\beta\ \forall\ \beta >\alpha 
\end{gathered}
\label{EKR.109}\end{equation}
identifying $G$-equivariant bundles over the original $G$-space $M$ with
iterated twisted bundles over the resolution of the quotient. Notice
however that the pull-back maps here involve additional information not found in the
quotient itself, and the same is true of the twistings.

There is a corresponding description of equivariant cohomology in terms of
iterated forms on the resolved quotient with coefficients in the bundles
formed from the localized representation rings of the isotropy groups.
These are linked by augmented pull-back maps. Then delocalized equivariant
cohomology is defined by analogy where the coefficient rings are the full
real representation rings. This results in an Atiyah-Hirzebruch isomorphism
\eqref{EKR.98}, as envisioned by Baum, Brylinski and MacPherson. This is
shown using the six term exact excision sequences corresponding to pruning the
isotropy tree, essentially as in \cite{EqKtRe}.

The authors thank Victor Guillemin, Eckhard Meinrenken and David Vogan for
helpful discussions.

\paperbody

\section{Resolution and Lifting}\label{Res}

We refer to the corresponding sections of \cite{EqKtRe} for a description
of the normal resolution of a compact manifold with a smooth action by a
compact Lie group and the categorical lifting of $G$-bundles. If the
initial manifold has corners then we demand that the action be `boundary
free' in the sense that if one connected boundary hypersurface is mapped to
another by an element of the group then they do not intersect. If this is
not the case then it can be arranged by replacing the manifold by its total
boundary blowup. This boundary-free property is preserved under normal
resolution and so applies throughout.

\section{Reduction}\label{Red}

Borel, see \cite{MR2189882}, showed how to reduce the geometry of the
action of a compact Lie group with a fixed isotropy type to the special
case of a fixed, i.e.\ normal, isotropy group. The geometry is thereby
reduced to that of a principal bundle for the quotient of the normalizer by
the chosen isotropy group. We need a more general version of this
construction which we term `reduction' below.

Let $X$ be a compact manifold (meaning possibly with corners always) with a
smooth action by a compact Lie group, $G,$ with unique isotropy type. This
is a fibre bundle $X\longrightarrow Y=X/G$ with fibre modeled on $G/H$ for
a choice of isotropy group, $H.$ By a reduction of $X$ we mean a smooth
submanifold $Z\subset X$ which meets every orbit, so defines a subfibration
over $Y,$ is stabilized by a closed subgroup $K$ in the strong sense that
\begin{equation}
g\in G,\ x\in Z\Mthen gx\in Z \Miff g\in K
\label{11.7.2020.3}\end{equation}
and on which $K$ acts with fixed isotropy type. In particular Borel's
reduction is the case where $Z=X^H$ is the set of points where the isotropy
group is a given element of the conjugacy class. Equivariant K-theory
behaves in the same way in this more general setting.

\begin{proposition}\label{11.7.2020.1} If $X$ is a compact manifold with
a smooth action, with unique isotropy type, by a compact Lie group, $G,$ and
$Z\subset X$ is a reduction of it, with stabilizer $K,$ as defined above, then
restriction gives an equivalence of categories 
\begin{equation}
\Bun_G(X)\to \Bun_K(Z).
\label{11.7.2020.2}\end{equation}
\end{proposition}

\begin{proof} If $V$ is a $G$-equivariant bundle over $X$ then its
  restriction to $Z$ is $K$-equivariant. This defines a map
  \eqref{11.7.2020.2} which is clearly functorial with respect to
  equivariant morphisms. Conversely if $U$ is a $K$-equivariant bundle
  over $Z$ then it lifts to $G\times Z$ under projection to the second
  factor. The left $G$-action induces a trivial $G$-action on the lifted
  bundle. The diagonal $K$-action in which $k$ maps $(g,z)$ to
  $(gk^{-1},kz)$ is free, commutes with the $G$-action and is covered by
  the action of $K$ on $U.$ By \eqref{11.7.2020.3} the quotient space given
  by the $K$-action is $X$ and the $K$-action on the bundle gives descent
  data defining a $G$-equivariant bundle. That these bundles and operations
  are smooth follows from the local product decompositions implied by the
  uniqueness of the isotropy types.

Clearly restriction is a left equivalence for this extension.
To see it is a right equivalence set $U=V\big|_{Z}$ for a $G$-equivariant
bundle $V$ on $X.$ Then the action on $V$ gives a bundle map of fibre
isomorphisms 
\begin{equation}
\{g\}\times V_z\longrightarrow V_{gz}
\label{11.7.2020.4}\end{equation}
which is equivariant for the $G$ actions and commutes with the $K$ action
$\{g\}\times V{z}\longrightarrow \{gk^{-1}\}\times V_{kz}$ and so extends
to an equivariant isomorphism of $V$ and the extension of $V\big|_{Z}.$
\end{proof}

Borel reduction corresponds to a choice among the isotropy groups $H\subset
G.$ The normalizer $N(H)=N_G(H)$ acts on $X^H$ with the subgroup $H$ acting
trivially and the quotient group $Q=N_G(H)/H$ acts freely. Thus $X^H$ is
a principal $Q$-bundle over the orbit space $Y=Q\backslash X^H$
\begin{equation}
\xymatrix{
Q\ar@{-}[r]&X^H\ar[d]\\
&Y
}
\label{E-K-n.285}\end{equation}
which is independent of the choice of $H$ up to equivalence.

However, the choice of a particular isotropy group is not natural, in general, and
at points below we need to allow changes of this choice. The closed
subgroups conjugate to $H$ in $G$ may be identified with the homogeneous
space $G/N(H)$ over which the space $X$ fibres with the fibre above
$H_g=gHg^{-1}$ being a principal $Q(H_g)=gQg^{-1}$-bundle. In this more
equivariant picture 
\begin{equation}
\xymatrix{
Q(H_g)\ar@{-}[r]&X\ar[d]\ar[r]^-{\iota}&G/N(H_g)\ar[dl]\\
&Y
}
\label{EKR.89}\end{equation}
the group $G$ still acts as a fibre-preserving map.

It is convenient to consider another setting closely related to
reduction. Thus, suppose that a compact Lie group $N$ acts on a compact
manifold with corners $Z$ with a fixed, hence normal, isotropy group $H.$ A
smooth submanifold $Z'\subset Z$ will be called a \emph{restriction} of the
action if it is a subbundle of $Z$ as a principal $N/H$ bundle with the
subgroup $N'\subset N$ acting fibre-transversally on $Z'$ with isotropy
group $H'=H\cap N'.$ Clearly there is then a restriction functor 
\begin{equation}
\Bun_N(Z)\longrightarrow \Bun_{N'}(Z').
\label{11.8.2020.1}\end{equation}
This functor arises in the connecting maps in the reduced picture of
equivariant K-theory below.

Borel's model for the equivariant cohomology of a $G$-space $X,$ for a
compact Lie group $G,$ is to take a classifying bundle for $G,$ so a
contractible space $EG$ on which $G$ acts freely. Then the product action
on $X\times EG$ is free. Writing the quotient space as $X\times_GEG$  
\begin{equation}
H_G(X)=H(X\times_GEG)
\label{8.10.2020.1}\end{equation}
where we are interested in the case of real coefficients.

\begin{proposition}\label{EKR.171} Under a reduction of a smooth action of
  a compact Lie group, as defined above
\begin{equation}
H_N(Z)=H_G(X)
\label{8.10.2020.5}\end{equation}
\end{proposition}

\begin{proof} Consider the product $G$ action on $G\times Z\times EG$ which
  commutes with the $N$ action recovering $X$ from $G\times Z.$ This has
  $G$ action has the transversal $\{e\}\times Z\times EG$ so the quotient
  may be identified with $Z\times EG.$ Projected to this transversal the
  action of $N$ on $G\times Z\times EG$ is identified with the product
  action on $Z\times EG$ with quotient $Z\times_NEG.$ On the other hand
  taking the quotient by the $N$ action first this is also identified with the
  quotient $X\times_GEG.$ Thus it follows that there is a natural
  identification \eqref{8.10.2020.5} for any reduction of a $G$ action.
\end{proof}

Note that for a reduction the image $gZ$ of $Z$ under any element $g\in G$ is a
reduction with group action by $gNg^{-1}$ and \eqref{8.10.2020.5} extends
to provide natural identifications between the various $H_{gNg^{-1}}(gZ).$

From the naturality of these constructions the $G$-equivariant Chern character
of a $G$ bundle over $X$ can be realized as the $N$-equivariant Chern
character for the restriction of the bundle to a reduction $Z$ and this is
natural with respect to restriction in the Cartan model for equivariant cohomology.

\section{Single Isotropy type}\label{Single}

The resolution of the manifold with group action reduces it to a tree of
compact manifolds with corners on each of which the group acts with all
isotropy groups conjugate; the resolved pieces are linked by equivariant
fibrations over the boundary hypersurfaces. Borel's reduction replaces each
component manifold $X$ by the submanifold $Z=X^H$ where the isotropy group,
$H,$ is fixed with the group replaced by the normalizer, i.e. $H\subset N$
is a normal subgroup. Then $Z$ is the total space of a principal
$Q=N/H$-bundle; by treating the components separately we can assume that the
base, $Y,$ is connected. We proceed to show how equivariant bundles are
reduced to twisted bundles over finite covers of the base.

Let $H'$ be the category of irreducible unitary representations of $H$ with
$M$ the tautological `bundle' over $H'$ with fibre the representation
space. The conjugation action of $N$ on $H$ induces an action of $N$ on
$H'$ with finite orbits, $o(\pi)=N\pi.$ Then over $Z\times o(\pi)$ consider
the trivial bundle obtained by pulling back the tautoligical bundle from
$H',$
\begin{equation}
M'\longrightarrow Z(\pi)=Z\times o(\pi ),
\label{EKR.112}\end{equation}
so with fibre at $\{z\}\times\{\pi '\}$ the representation space $M(\pi')$ of $\pi
'\in o(\pi ).$ Each $n\in N$ conjugates the representation $\pi'$ to $n\pi
'$ and by Schur's Lemma the unitary conjugation is determined up to a
multiple of the identity. This determines a central extension
\begin{equation}
\UU(1)\longrightarrow \vN(\pi)\longrightarrow N.
\label{EKR.76}\end{equation}
Thus $\vN(\pi)$ acts equivariantly on $M'$ covering its action through
$N$ on $Z(\pi).$ We write $\vN(\pi)$ since it is the `opposite' extension, $\hN(\pi),$
given by the dual of the circle bundle \eqref{EKR.76}, which appears in the
decomposition of equivariant bundles.

The action by $N$ on $Z(\pi)$ has isotropy group $H$ which naturally
includes in $\vN(\pi)$ since it acts consistently on each $M(\pi).$ Thus
$Z(\pi)$ is the total space of a principal bundle with base $Y(\pi)$ and
structure group $Q=N/H$ whereas $M'$ is a bundle over $Z(\pi)$ with an
equivariant action of $\vN(\pi).$

We now turn to the decomposition of an $N$-equivariant bundle, $V,$ over $Z$ on
which $N$ acts with fixed isotropy group $H.$ The pointwise action of $H$
induces a finite decomposition into subbundles
\begin{equation}
V=\bigoplus_{\pi \in H'} V(\pi)
\label{EKR.73}\end{equation}
where the action of $H$ on $V(\pi)$ is conjugate to the action on
$\bbC^n\otimes M(\pi )$ for some $n.$ This can be formalized by considering the bundle
over $Z\times H'$ with fibre at $(z,\pi')$ the space of $H$-equivariant homomorphisms
\begin{equation}
W(z,\pi')=\{l:M(\pi')\longrightarrow V_z;l\circ \pi'(h)=h\circ l\}=\hom_H(M(\pi'),V_z,).
\label{EKR.115}\end{equation}
The rigidity of representations implies that these are smooth bundles,
although the dimension of $W(z,\pi')$ may vary between components of $Z.$

\begin{proposition}\label{EKR.116} The fibres \eqref{EKR.115} define a
  smooth bundle bundle $W$ over $Z\times H'$ which carries an equivariant
  action of $\hQ(\pi)=\hN(\pi)/H$ on the restriction $W(\pi)$ to each orbit
  $Z(\pi)$ and is such that summing over the fibres
\begin{equation}
V=(\pi_Z)_*(W\otimes M')
\label{EKR.127}\end{equation}
as $N$-equivariant bundles over $Z.$ This defines an equivalence of
categories 
\begin{equation}
\Bun_N(Z)\longrightarrow \bigoplus_{\pi\in H'/N}\Bun_{\hQ(\pi)}(Z(\pi)).
\label{EKR.128}\end{equation}
\end{proposition}

\begin{proof} The given equivariant bundle $V$ decomposes as in
  \eqref{EKR.73} and so lifts to a bundle $V'$ over $Z\times H'$ where the
  fibre at $(z,\pi)$ is the image of $W(z,\pi).$ Then 
\begin{equation}
W=\hom_H(M',V')=V'\otimes_H(M')^*
\label{EKR.129}\end{equation}
is the bundle of $H$-equivariant homomorphisms over $Z\times H'.$ Since the
$\vN(\pi)$ action on $M'$ induces an $\hN(\pi)$ action on the dual, $W$ has
(over each orbit $Z(\pi))$ an action of $\hN(\pi)$ on the second
factor and of $N$ on the first. Combined this gives an equivariant action
of $\hN(\pi)$ in which $H$ acts trivially. Thus $W$ carries an equivariant
action of $\hQ(\pi).$ Then $V'$ is recovered as the bundle $W\otimes M'$
over $Z\times H'$ in which the extensions in the $\hN(\pi)$ and $\vN(\pi)$
actions cancel to give an $N$-equivariant bundle which pushes forward,
i.e.\ sums over the finite number of representations present, to give $V.$ 

This relationship is an equivalence of categories.
\end{proof}

If $W$ was actually $Q$-equivariant its restriction to $Z(\pi)$
would descend to a bundle over the base, $Y(\pi),$ for the action of $Q.$
Since the action is only projective it becomes a twisted bundle over
$Y(\pi).$ We treat these in the context of \emph{gerbe modules} which we
now recall. The central extension of $Q(\pi)$ induces a lifting bundle
gerbe, in the sense of Murray \cite{Murray}, over $Y(\pi):$
\begin{equation}
\xymatrix{
&&C(\pi)\ar[d]\ar@{-->}[r]&\hat Q(\pi)\ar[d]\\
Q(\pi)\ar@{-}[r]&Z\ar[d]&\ar@<2pt>[l]\ar@<-2pt>[l]
Z^{[2]}\ar[dl]\ar[r]&Q(\pi)\\
&Y(\pi)
}
\label{E-K-n.274}\end{equation}
where $C(\pi)$ is the pull-back of the circle giving the central extension.

This gerbe is classified by its (torsion) Dixmier-Douady class
$\alpha(\pi)\in H^3(Y(\pi);\bbZ)$ which in this case is the transgression
of the Mackey class of the extension. A vector bundle $W(\pi)$ over
$Z(\pi)$ with a $\hat Q(\pi)$-action covering the $Q(\pi)$-action is a
gerbe module in the sense of \cite{MR1911247} since the induced
isomorphisms over each point $(p,q)\in Z^{[2]}$
\begin{equation}
\gamma _{p,q}:W_p\otimes C(\pi)_{p,q}\longrightarrow W_q
\label{E-K-n.277}\end{equation}
are consistent with the gerbe structure. Thus:

\begin{proposition}\label{E-K-n.276} A $G$-equivariant vector bundle on
$X,$ for an action with unique isotropy type, induces a gerbe module over
$Y(\pi)$ for each of the lifting bundle gerbes \eqref{E-K-n.274}, trivial
outside a finite subset of $H'/N$ and conversely; $G$-equivariant bundle
isomorphisms correspond uniquely to $\hat Q(\pi)$-equivariant bundle
isomorphisms, and hence gerbe module isomorphisms resulting in an
equivalence of categories
\begin{equation}
\Bun_G(X)=\bigoplus_{\pi\in H'/N}\Bun(Y(\pi);\alpha (\pi )).
\label{EKR.118}\end{equation}
\end{proposition}

Thus, each $G$-equivariant bundle over $X$ induces gerbe modules over a
finite collection of the lifting gerbes \eqref{E-K-n.274} with $\pi$ lying
in different orbits for the action of $N$ on $H'.$ Different choices of
base point in the orbits induce gerbe module isomorphisms.

\begin{theorem}[A. Wassermann]\label{E-K-n.275} For the action of a compact
  group $G$ on a compact manifold $X$ with unique isotropy type, the
  equivariant K-theory can be identified with the collective twisted
  K-theory with coefficients in the representation ring of the isotropy group
\begin{equation}
K_G(X)=K(Y(\pi);\alpha(\pi))\otimes_{H'/N} \cR(H)
\label{E-K-n.278}\end{equation}
\end{theorem}

\begin{proof} This is a direct consequence of Proposition~\ref{E-K-n.276},
  by passing to the corresponding Grothendieck groups, and the
  characterization of twisted K-theory in terms of gerbe modules by
  Bouwknegt, Carey, Varghese, Murray and Stevenson in \cite{MR1911247}.
\end{proof}

Although \eqref{EKR.118} gives a complete description of equivariant
bundles in this case, the twisted bundles can be further reduced, in the
sense of \S\ref{Red}, and this is relevant to the discussion of pull-back
maps below.

In the action of $N$ on $H'$ for a normal subgroup $H,$ consider the
stabilizer $N(\pi)$ and its normalizer $\Gamma (\pi ).$ The submanifold 
\begin{equation}
Z'(\pi )=Z\times \Gamma (\pi )\pi \subset Z(\pi )=Z\times N\pi  
\label{EKR.119}\end{equation}
is a reduction of $Z(\pi)$ for the action of $\hQ(\pi)$ which is reduced to
the action of $\hQ'(\pi )=\hGa(\pi )/H$ where $\hGa(\pi)$ is the central
extension of $\Gamma(\pi)$ induced by $\hN(\pi).$ 

As noted in the Introduction, and proved already by Cartan, the equivariant
cohomology for a free action is the cohomology of the base. This can be
seen directly in terms of the Borel model and there is a corresponding
argument for the Cartan model using a connection, see the book of Guillemin
and Sternberg, \cite{Guillemin-Sternberg}. Namely if $\alpha$ is a
connection on $X$ as a principal $G$-bundle over $Y$ then the trace
functional on $\mathfrak{g}\otimes \mathfrak{g}^*$ defines a linear map
\begin{equation}
\CI(X;\Lambda ^*\otimes S(\mathfrak{g}^*))\ni u\longrightarrow
Au=\Tr_{\mathfrak{g}}(\alpha \wedge u)\in\CI(X;\Lambda ^*\otimes S(\mathfrak{g^*})).
\label{EKR.154}\end{equation}
Here $S(\mathfrak{g}^*)$ is realized as the totally symmetric polynomials
with the trace acting in the first variable. This lowers the polynomial
order by one. Acting on $G$ -invariant forms which are homogeneous of degree $p$ as
polynomials on $\mathfrak g$ 
\begin{equation}
d_G(Au)-Ad_G+pu=pB_pu
\label{EKR.155}\end{equation}
where $B_p$ is given in terms of contraction with the curvature 2-form of
$\alpha$ and is of degree $p-1$ as a polynomial. Repeated application of
$B_p$ to the terms of highest positive degree, eventually stabilizes on
each form, replacing an equivariant form by a basic one as discussed by
Meinrenken \cite{Meinrenken-Enc}. 

The same approach can be used to study the Cartan cohomology in the case of
an action with single isotropy type. It follows from the relationship
between the Borel and Cartan models, see \cite[Chap 4]{Guillemin-Sternberg},
that under Borel reduction the restriction map
\begin{equation}
\CI(X;\Lambda ^*\otimes S(\mathfrak{g}^*))^G\longrightarrow
\CI(X^H;\Lambda ^*\otimes S(\mathfrak{n}^*))^N
\label{EKR.156}\end{equation}
induces an homotopy equivalence of complexes realizing the isomorphism of
$H^*_G(X)$ and $H^*_N(X^H)$ where $N=N(H)$ is the normalizer.

To see directly that the restriction, both on the manifold and the Lie
algebra, gives an isomorphism the contraction operator $A$ in
\eqref{EKR.154} given in terms of a connection form, can be replaced by
contraction with a `partial connection'. For convenience, give the group an
invariant metric. By such a partial connection form we mean a $G$-equivariant
1-form on $X,$ with values in $\mathfrak g,$ which restricted to the
tangent space to the orbit at each point restricts to the identity on the
orthocomplement of the Lie algebra of the normalizer of the isotropy group
at that point. For the homogeneous space $G/H$ such a form can be
constructed by orthogonal projection of the Maurier-Cartan form and in
general by using a $G$-invariant metric on $X$ to project the orthogonal
subspace to the Lie algebra of the normalizer in the tangent space to the
orbits. Application of the analogue of \eqref{EKR.154} lowers the
polynomial order pointwise in this subspace and gives a homotopy
equivalence for corresponding to \eqref{EKR.156}.

In the reduced case written as the action by $N$ on $Z$ with fixed isotropy
group $H$ the conjugation action of $N$ on the Lie algebra $\mathfrak{h}$
induces a locally trivial equivariant action on $Z\times
(S^*(\mathfrak{h^*})^H)$ defining a flat bundle $S_H$ over the base.

\begin{proposition}\label{EKR.158} If $N$ acts on a compact manifold $Z$
  with unique isotropy group $H$ and base $Y,$ a choice of invariant inner product
  on $N$ and a connection on $Z$ as the total space of a principal $Q=N/H$ bundle
  defines a homotopy equivalence intertwining the equivariant deRham
  differential on $Z$ and the deRham differential on $Y$ with coefficients
\begin{equation}
B^\infty :\left(\CI(Z;\Lambda ^*\otimes S(\mathfrak{n}^*)\right)^N\longrightarrow
\CI(Y;\Lambda
^*\otimes S_H).
\label{EKR.159}\end{equation}
\end{proposition}

\begin{proof} The given connection form $\alpha\in\CI(Z;\Lambda
  ^1\otimes\mathfrak{q})$ for $Z$ as a principal $Q$ bundle can be
  interpreted as a form with values in $\mathfrak{n}$ using the embedding
  of $\mathfrak{q}\subset\mathfrak{n}$ as the space orthogonal to
  $\mathfrak{h}.$ Then the identity \eqref{EKR.155} allow a successive
  lowering of the polynomial order in $\mathfrak{q}$ of Cartan forms on $Z$
  so identifying the equivariant cohomology with the deRham cohomology of
  basic $Q$-invariant sections of $\CI(Z;\Lambda ^*\otimes
  (S^*(\mathfrak{h^*})^H)).$ The action of $Q$ is locally trivial and
  defines the descent to the cohomology of $Y$ with coefficients in $S_H.$
\end{proof}

For the extension to a general action below it is important that
\eqref{EKR.159} defines an explicit homotopy equivalence.

In this reduced case of an action with single isotropy group, an
equivariant bundle $V$ decomposes as a $\hQ(\pi)$-equivariant bundle
$W(\pi)$ over each $Z(\pi).$ the Chern character can be see directly in
terms of Proposition~\ref{EKR.158}. Namely, given an equivariant
connection, the $\hQ(\pi)$-equivariant Chern character of $W(\pi)$ is a
$Q(\pi)$-equivariant form on $Z(\pi).$ The character map given by
the trace defines a section of the trivial polynomial bundle
\begin{equation}
\tr:H'\longrightarrow (S(\mathfrak h^*))^H
\label{EKR.164}\end{equation}
over $Z(\pi)$ which is invariant under the action of $Q$
projected from the action of $\vQ(\pi)$ on $M'.$ The form and section
combine to give a form over the quotient, $Y(\pi)$ with values in the flat
descended bundle and this further pushes forward, summing over the fibres
of each $Y(\pi),$ and over $H'/N,$to give the closed form
\begin{equation}
\Ch(V)\in\CI(Y;\Lambda ^{\ev}\otimes S_H).
\label{EKR.165}\end{equation}
This descends to the equivariant Chern character in $H_N(Z)=H_G(X).$

The \emph{delocalized} equivariant cohomology in this case of a single isotropy group
is obtained by replacing the character ring by the real representation ring
$\cR(H).$ As a vector space this is freely generated by $H'.$ Thus there is
a linear subspace $o(\pi)\otimes\bbR$ associated with each orbit of the
action of $Q$ on $H'$ which can be viewed as the push forward under the map
from $o(\pi)$ to a point. The action of $Q$ on $Z\times o(\pi)$ therefore
generates a flat real line bundle over $Y(\pi)$ which pushes forward to $Y$
to define a flat bundle. The formal sum over $H'/N$ of these spaces is the
flat bundle $\cR_H(Y).$

\begin{definition}\label{EKR.167} The
\emph{delocalized equivariant cohomology,}
$H^*_{G,\dl}(X),$ of a manifold with smooth action, with unique isotropy
type, by a compact Lie group is the deRham cohomology of the base with
coefficients in $\cR_H(Y).$
\end{definition}

Note that the forms appearing here, in $\CI(X/G;\Lambda ^*\otimes
\cR_H(Y)),$ are each required to have coefficients in a fixed finite span of the
subspaces corresponding to the orbits of the normalizer of an isotropy
group on its representation ring. This space of forms is functorial under
change of isotropy group as in \eqref{EKR.89}.

\begin{theorem}[Following \cite{MR86g:55006}]\label{EKR.168} For the action
  of a compact Lie group $G$ on a compact manifold $X$ with unique isotropy
  type, the equivariant Chern character factors through a natural
  `delocalized Chern character' as in \eqref{EKR.98}
\begin{equation}
\xymatrix{
K^0_G(M)\ar[rr]^{\Ch_{\dl}}\ar[dr]_{\Ch}&&H^{\ev}_{G,\dl}(M)\ar[dl]^{\loc}\\
&H^{\ev}_{G}(M)
}
\label{EKR.169}\end{equation}
with the top map an isomorphism after tensoring with $\bbR.$
\end{theorem}

\noindent So this is an Atiyah-Hirzebruch isomorphism.

\begin{proof} As in the discussion of the standard equivariant Chern
  character above, an equivariant K-class is represented by a formal
  difference $V\ominus\bbC^N$ where the trivial bundle contributes a
  constant form. The decomposition \eqref{E-K-n.298} yields the bundle
  $W(\pi)$ over each of the principal $Q$ bundles $Z(\pi).$ A
  $\hQ$-invariant connection on $W(\pi)$ yields a $Q$-invariant form on
  $Z(\pi).$ Tensored with the canonical section on the trivial
  $\bbR$-bundle over $H'$ with its $Q$ action this descends to a form with
  values in the flat line bundle over the quotient $Y(\pi)$ and then pushes
  forward to a closed element of $\CI(Y;\Lambda ^{\ev}\otimes \cR_H.)$ The
  deRham class of this form is well-defined, giving the map $\Ch_{\dl}$ in
  \eqref{EKR.169}. The trace map on representations induces the
  localization map completing the commutative diagram.

For torsion-twisted K-theory there is an Atiyah-Hirzebruch isomorphism 
\begin{equation}
\Ch:K(Y;\alpha )\otimes\bbR\longrightarrow H^{\ev}(Y;\bbR).
\label{EKR.182}\end{equation}
The crucial injectivity of this map follows from the multiplicative
property of twisted K-theory  
\begin{equation}
K(Y;\alpha )\times K(Y;\beta)\longrightarrow K(Y;\alpha +\beta )
\label{EKR.183}\end{equation}
(given consistent trivializations of the torsion classes). At the level of
bundles \eqref{EKR.182} corresponds to the fact that if $n\alpha =0$ then
the $n$-fold tensor product of a twisted bundle is untwisted and division
by $n$ is an isomorphism over $\bbR.$

Applying \eqref{EKR.182}, at the level of bundles over the $Y(pi),$ to
\eqref{EKR.118} shows the injectivity of the delocalized Chern character
over $\bbR.$ Since the surjectivity of the usual equivariant Chern
character is known the delocalized Atiyah-Hirzebruch isomorphism follows.
\end{proof}

\section{Equivariant fibrations}\label{Fib}

Consider two compact manifolds $X_i,$ $i=0,1,$ with smooth $G$
actions with fixed isotropy type where $X_0$ is the total space of a
$G$-equivariant fibre bundle over $X_1$
\begin{equation}
\xymatrix{
G\ar[d]_{\Id}\ar@{-}[r]&\ar[d]^{\phi} X_0\ar[r]&Y_0\ar[d]^{\tphi}\\
G\ar@{-}[r]& X_1\ar[r]&Y_1
}\text{with }Y_0\text{ connected.} 
\label{EKR.90}\end{equation}
In the absence of the connectedness condition on $Y_0$ (which becomes
significant in the stability of subgroups) we proceed component by
component.

First choose an isotropy group $H_1$ for the lower action and pass to the
Borel reduction $Z_1=X_1^{H_1}\subset X_1$ with its $N_1=N(H_1)$
action. Next consider the preimage
\begin{equation}
Z_0=\phi^{-1}(X_1^{H_1})\subset X_0.
\label{EKR.93}\end{equation}
As the preimage of a smooth manifold under a fibration this is smooth and
meets each orbit of the $G$-action on $X_0.$ Furthermore the pointwise
stabilizer is exactly $N_1$ since if $x\in \phi^{-1}(X_1^{H_1})$ and $gx\in
\phi^{-1}(X_1^{H_1})$ then $\phi(x)\in X_1^{H_1}$ and $g\phi(x)=\phi(gx)\in
X_1^{H_1}$ so $g\in N_1$ and conversely. The isotropy groups at points of
$Z_0$ remain unchanged since they are necessarily subgroups of $H_1\subset
N_1.$ Thus $Z_0$ is a reduction of the action on $X_0.$

We have thereby reduced the equivariant fibration to the special case
where the lower action has unique isotropy group $H_1$
\begin{equation}
\xymatrix{ N_1\ar[d]_{\Id}\ar@{-}[r]&\ar[d]^{\phi}
  Z_0\ar[r]&Y_0\ar[d]^{\tphi}\\ N_1\ar@{-}[r]&
  Z_1\ar[r]&Y_1&\kern-10pt\text{ with normal isotropy group }H_1.}
\label{EKR.122}\end{equation}

Now, consider the components of $Z_0.$ Since $Y_0$ is connected each of
these connected submanifolds meets every orbit in $Z_0$ and hence projects
onto $Y_0.$ Choose a point in $Z_0$ at which the isotropy group is $H_0.$
The rigidity of subgroups, already utilized above, shows that within the
component containing this point all isotropy groups fall in the
$H_1$-conjugacy class of $H_0$ as a subgroup of $H_1.$ Indeed, near the
chosen point, the isotropy groups are contained in $H_1$ and are conjugate
to $H_0$ by an element of $N_1$ near the identity and hence by the product
$n_0h_1$ with $n_0$ in the normalizer of $H_0$ and $h_1\in H_1.$ Thus they
are $H_1$-conjugate to $H_0,$ in fact by an element in the component of the
identity in $H_1.$ This therefore remains true over the component of $Z_0.$
Let $Z_0'$ be the union of the images of this component under the action of
$H_1;$ all the isotropy groups on $Z_0'$ are $H_1$-conjugate to $H_0$ and
all such groups occur in each orbit. The stabilizer of $Z'_0$ is the
subgroup $N_1'\subset N_1$ (of finite index) fixing the $H_1$-isotropy class
of $H_0.$ This is a reduction of the group action on $Z_0$ and hence of
that on $X_0.$

The image $\phi(Z_0')=Z_1'$ is a reduction of the action on $Z_1$ to the
subgroup $N_1';$ each of the components of $Z'_0$ projects to a smooth
submanifold and for two components the images are either equal or
disjoint. Thus $Z_1'\subset Z_1$ is fixed by the condition that at every
point in the preimage in $Z_0$ the isotropy group is $H_1$-conjugate to
$H_0.$ Thus we arrive at the reduced equivariant fibration
\begin{equation}
\xymatrix{ N'_1\ar[d]_{\Id}\ar@{-}[r]&\ar[d]^{\phi} Z'_0\ar[r]&Y_0\ar[d]^{\tphi}
&\kern-25pt\text{ with isotropy groups }H_1\text{-conjugate to }H_0\\
N'_1\ar@{-}[r]&
  Z'_1\ar[r]&Y_1&\kern-60pt\text{ with normal isotropy group }H_1.}
\label{EKR.123}\end{equation}

From this we deduce

\begin{proposition}\label{EKR.91} For an equivariant fibration
\eqref{EKR.90} with $Y_0$ connected there are choices of isotropy groups
$H_0\subset H_1$ for the two actions and a (`\emph{double}') reduction,
$Z''_0,$ of $X_0$ with fixed isotropy group $H_0$ and group action of
$N_0=N_G(H_0)\cap N_G(H_1)$ and $Z_1'$ of $X_1$ with isotropy groups $H_1$
and group action by $N_1'=N_0H_1$ giving an equivariant fibration
\begin{equation}
\xymatrix{
N_0\ar@{^(->}[d]\ar@{-}[r]&\ar[d]^{\phi} Z''_0\ar[r]&Y_0\ar[d]^{\tphi}
&\kern-25pt\text{with normal isotropy group }H_0\\
N_1'\ar@{-}[r]& Z'_1\ar[r]&Y_1&\kern-25pt\text{ with normal isotropy group }H_1.}
\label{EKR.92}\end{equation}
\end{proposition}

\begin{proof} Having arrived at \eqref{EKR.123} we make the Borel reduction
  of the upper action corresponding to the isotropy group $H_0$ and with
  action by $N_0,$ the normalizer of $H_0$ in $N_1'.$ The fibres of $\phi$
  above points of $Z_1'$ have actions by $H_1$ and isotropy groups, for
  the $N_1'$ action on $Z_1',$ $H_1$ conjugate to $H_0.$ Thus, each fibre
  meets $Z_0''$ which fibres over $Z_1'.$ The fact that the
  $N_1'$-conjugates of $H_0$ are also $H_1$-conjugates means that for each
  $n_1\in N_1'$ there exists $h_1\in H_1$ such that 
\begin{equation}
n_1H_0n_1^{-1}=h_1H_0h_1^{-1}\Longrightarrow
h_1^{-1}n_1H_0(h_1^{-1}n_1)^{-1}=H_0\Longrightarrow N_1'=H_1N_0.
\label{EKR.124}\end{equation}
Since $H_1$ is normal in $N_1',$ $N_1'=N_0H_1$ and $H_1$ acts trivially on $Z_1'.$
\end{proof}

Note that these reductions occur precisely because the principal bundles
corresponding to the $G$ actions are partially trivialized by the existence
of the equivariant fibration.  Since $Z''_0$ is a reduction of the $G$
action on $X_0$ and $Z'_1$ is a reduction for $X_1$ there are, by
Proposition~\ref{11.7.2020.1}, equivalences of categories and hence a
pull-back functor $\phi^\#$
\begin{equation}
\xymatrix{
\Bun_G(X_0)\ar[r]& \Bun_{N_0}(Z''_0)\ar[r]
&\bigoplus_{\pi\in H'_0/N_0}\Bun_{\hQ_0(\pi)}(Y_0(\pi))\\
\Bun_G(X_1)\ar[r]\ar[u]_{\phi^*}& \Bun_{N'_1}(Z'_1)\ar[u]_{\phi^\#}\ar[r]
&\ar[u]_{\phi^\#}\bigoplus_{\sigma\in H'_1/N_1'}\Bun_{\hQ_1'(\sigma)}(Y_1(\sigma))
}
\label{EKR.95}\end{equation}
which we proceed to describe more explicitly.

The pull-back operation on the right in \eqref{EKR.95} is the composite of
pull-back for $N_1'$-equivariant bundles and restriction to the Borel
reduction to $N_0$-equivariant bundles. It follows that it can also written
as the composite of restriction, from $N_1'$-equivariant to
$N_0$-equivariant bundles on $Z_1'$ followed by pull-back of
$N_0$-equivariant bundles. The functoriality of reduction to twisted
bundles shows that the second step is just pull-back of twisted bundles so
we concentrate on the first step which can be understood in terms of
`branching maps'.

For a compact Lie group and closed subgroup $H_0\subset H_1$ consider the
bundle over the product of the sets of unitary irreducibles 
\begin{equation}
\xymatrix{
\tau\ar[r]& H_1'\times H_0'
}
\label{EKR.126}\end{equation}
with fibre at $(\sigma ,\pi)$ the space $\hom_{H_0}(\pi ,\sigma)$ of
$H_0$-equivariant linear maps from the representation space $M_0(\pi)$ of $\pi$
to the representation space $M_1(\sigma)$ of $\sigma.$ The support of this
bundle is proper as a relation. 

For simplicity we will denote $N'_1$ in the discussion by $N_1.$
Thus suppose as above that $H _1\subset N_1$ is a normal subgroup of a compact Lie
group, $H_0\subset H_1,$ and $N_0\subset N_1$ are closed subgroups with
$H_0$ normal in $N_0$ and $N_1=H_1\cdot N_0.$ As a subgroup
$N_0\subset N_1$ acquires a central extension
$\vN_0(\sigma)\subset\vN_1(\sigma)$ from the central extension of $N_1$
associated to the orbit of $\sigma \in H_1'.$ The tautological bundle
$M_1'$ over $H_1'$ therefore has an action of $\vN_0(\sigma)$ when
restricted to $N_0\sigma.$ In consequence the bundle $\tau$ has an action of
$\hN_0(\pi)\times \vN_1(\sigma)$ over $N_0\sigma \times N_1\pi.$ 

\begin{proposition}\label{EKR.125} If $X_1$ is a compact manifold with
  smooth action by $N_1$ with fixed isotropy group $H_1$ and $H_0,$ $N_0$
  are as above then for each pair $(\sigma,\pi)\in H_1'\times H_0$ there is
  a `branching map'
\begin{equation}
\tau:\Bun(Y(\sigma);\alpha (\sigma))\longrightarrow \Bun(Y(\pi);\alpha (\pi))
\label{EKR.130}\end{equation}
from $\hN_1(\sigma)$-equivariant bundles over $(X_1\times
N_1\sigma )/N_1$ to $\hN_0$-equivariant bundles over $(X_1\times
N_0\pi)/N_0$ given by tensor product with $\tau$ pulled back to $X_1.$
\end{proposition}

\begin{corollary}\label{EKR.131} 
The pull-back map on the right in \eqref{EKR.95} is the composite  
\begin{equation}
\phi^\#=\phi^*\circ\tau
\label{EKR.132}\end{equation}
where $\phi^*$ is the pull-back of bundle gerbe modules covering the
pull-back of principal bundles.
\end{corollary}

For the Cartan model for equivariant cohomology, as noted in \eqref{EKR.156},
reduction corresponds to an homotopy equivalence to the complex of equivariant
forms on the reduction. Thus under the multiple reductions on $X_0,$  
\begin{equation}
(\CI(X_0;\Lambda ^*\otimes S(\mathfrak{g}^*))^G\longrightarrow
  (\CI(Z''_0;\Lambda ^*\otimes S(\mathfrak{n}_0^*))^{N_0}
\label{EKR.173}\end{equation}
where $\mathfrak{n}_0$ is the Lie algebra of $N_0$ which normalizes the
isotropy group $H_0.$ The quotient group $Q=N_0/H_0$ acts freely on $Z''_0$
and \eqref{EKR.159} gives the further homotopy equivalence 
\begin{equation}
(\CI(Z''_0;\Lambda ^*\otimes S(\mathfrak{n}_0^*))^{N_0}\longrightarrow
  \CI(Y_0;\Lambda ^*\otimes S_{H_0})
\label{EKR.174}\end{equation}
to the deRham complex with coefficients in the flat bundle $S_{H_0}$ over $Y_0.$

Similarly in the base the successive reductions and `twisted reduction' give
homotopy equivalences 
\begin{equation}
(\CI(X_1;\Lambda ^*\otimes S(\mathfrak{g}^*))^{G}\longrightarrow
  (\CI(Z'_1;\Lambda ^*\otimes S(\mathfrak{n}_1^*))^{N'_1}\longrightarrow
  \CI(Y_1;\Lambda ^*\otimes S_{H_1})
\label{EKR.175}\end{equation}
to the deRham complex with coefficients in $S_{H_1}$ which has typical fibre
$S(\mathfrak{n}^*_1).$  Since $Q=N'_1/H_1=N_0/H_0,$ the actions on
$S(\mathfrak{n}_0^*)$ and $\mathfrak{n}_1^*$ are consistent with
restriction, induced by the inclusion of $N_0\hookrightarrow N_1'$ and
hence there is a natural pull-back and restriction map covering $\phi:$
\begin{equation}
(S_{H_1})_{\phi(y_0)}\longrightarrow (S_{H_0})_{y_0},\ y_0\in Y_0
\label{EKR.176}\end{equation}
and hence an augmented, smooth, pull-back map on `deRham' sections 
\begin{equation}
\phi^{\#}:\CI(Y_1;\Lambda ^*\otimes S_{H_1})\longrightarrow \CI(Y_0;\Lambda
^*\otimes S_{H_0}).
\label{EKR.177}\end{equation}
\begin{proposition}\label{EKR.178} Pull-back and iterated reduction gives a
  commutative diagram 
\begin{equation}
\xymatrix{
(\CI(X_0;\Lambda ^*\otimes S(\mathfrak{g}^*))^{G}\ar[r]& \CI(Y_0;\Lambda ^*\otimes S_{H_0})\\
(\CI(X_1;\Lambda ^*\otimes S(\mathfrak{g}^*))^{G}\ar[u]^{\phi^*}\ar[r]& \CI(Y_1;\Lambda ^*\otimes S_{H_1})\ar[u]_{\phi^\#}
}
\label{EKR.179}\end{equation}
where the horizontal maps are homotopy equivalences.
\end{proposition}

Pull-back for the forms corresponding to delocalized equivariant cohomology
behaves similarly. The quotient bundle $Q$ acts on both $H'_1$ and $H'_0.$
As in the case of a general equivariant bundle the tautological bundle
$M_1'$ of $H_1$-representations over $Z_1'\times H_1'$ decomposes into a
bundle over the product with $H_0'$
\begin{equation}
M'_{10}\longrightarrow Z_1'\times H_1'\times H_0'
\label{EKR.180}\end{equation}
with proper support, such that $M'_{10}\otimes M'_0$ pushes forward off
$H_0'$ to $M_1'$ over $Z_1'\times H_1'.$ Restricting to a $Q$-orbit
$o_1(\sigma )\subset H_1'$ the support of $M'_{10}$ is restricted to the
finite number of orbits $o_0(\pi)\subset H_0'$ of $H_0$-subrepresentations
of $\sigma.$ The bundle $M'_1$ over $Z_1'\times o_1(\sigma )$ has an action
of the corresponding central extension $\vQ(\sigma)$ and following the
argument above $M'_{10}$ has an action by the central extension $\vQ(\pi
,\sigma)$ which has the property 
\begin{equation}
\vQ(\pi ,\sigma )\otimes \vQ(\sigma )=\vQ(\pi )
\label{EKR.181}\end{equation}
with the tensor product referring to the circle bundles giving the central
extensions. 

\section{Reduced models}\label{Mod}

Each of the three equivariant cohomology theories, K-theory, delocalized
(by definition) and Cartan, has a model over the iterated space resolving
the quotient by the group action and at this level the behavior of the
Chern character can be seen rather directly.

Not surprisingly, the most complicated of these is equivariant
K-theory. Each of the strata, $X_\alpha ,$ of the normal resolution, has
Borel reduction to a principal bundle with total space $Z_\alpha $ and base
$Y_\alpha$ for the action of $Q_\alpha =N(H_\alpha)/H_\alpha$ where
$H_\alpha$ is a choice of isotropy group. We need to choose the isotropy
groups consistently to define the pull-back maps so we proceed
categorically, allowing all possible choices of isotropy group. Once a
choice is made the spaces
\begin{equation}
Y_\alpha (\pi)=Z_\alpha(\pi)/N(H_\alpha ),\ \pi\in
H_\alpha ',\  X_\alpha(\pi)=X_\alpha \times o(\pi),\ o(\pi)=N(H_\alpha)\pi
\label{EKR.133}\end{equation}
are defined for each orbit of $N(H_\alpha)$ acting on the category of
irreducible unitary representations $H'_\alpha.$ Then the reduced version
of an equivariant bundle consists of a finite collection of twisted bundles
for the Dixmier-Douady classes, $a_\alpha(\pi),$ of $Z_\alpha (\pi)$ over the
corresponding $Y_\alpha(\pi ).$ As discussed above, these are bundles
gerbes over the $X_\alpha (\pi)$ as the total space of a lifting bundle
gerbe given by the free action of $Q_\alpha (\pi)$ and the central
extension $\hQ_\alpha (\pi)$ arising from the action of $N(H_\alpha)$ on
$o(\pi).$ Thus for each stratum we have a category of reduced equivariant bundles
\begin{equation}
\Bun_\alpha=\bigoplus_{H'_\alpha /N(H_\alpha )} \Bun(Y_\alpha (\pi),a_\alpha (\pi)).
\label{EKR.134}\end{equation}
This behaves naturally under change of choices.

For each $\beta <\alpha$ there is a boundary hypersurface
$H_{\beta,\alpha}$ of $X_\beta$ with fibration $\phi_{\beta ,\alpha
}:H_{\beta ,\alpha }\longrightarrow X_\alpha.$ Proposition ~\ref{EKR.130}
describes the corresponding augmented pull-back map
\begin{equation}
\phi_{\beta ,\alpha }^{\#}:\Bun_\alpha\longrightarrow
\Bun_\beta \big|_{H_{\beta ,\alpha }}.
\label{EKR.135}\end{equation}

Then the `chain space' for K-theory consists of $W_\alpha \in\Bun_\alpha$
for each $\alpha$ such that for all $\alpha <\beta$ 
\begin{equation}
W_\alpha \big|_{H_{\beta ,\alpha }}\simeq \phi^\#_{\beta ,\alpha }W_\beta 
\label{EKR.136}\end{equation}
where the implied isomorphisms must compose under iteration. The collection
of reduced bundles is then a category and the reduced K-theory $K_{\red}(Y_*)$ is the
corresponding Grothendieck group; of course, despite the notation, this
depends on more information than just the quotient spaces.

\begin{proposition}\label{EKR.137} The equivariant K-theory of $X$ is
  naturally isomorphic to the reduced K-theory $K_{\red}(Y_*)$ of the
  resolution of the quotient.
\end{proposition}

\begin{proof} This follows by application of the reduction results above.
\end{proof}

For equivariant cohomology there is a closely related `reduced' Cartan
model. In the same setting as above, let $C(H_\alpha)$ be the space of
$H_\alpha$-invariant polynomials on the dual of the  Lie algebra of $H_\alpha.$ The
normalizer $N(H_\sigma)$ acts on this space by conjugation, with the action
descending to $Q_\alpha$ and locally trivial. Lifting $C(H_\alpha)$ to a
trivial bundle over $X_\alpha$ this action defines a flat bundle $C_\alpha$
over the base $Y_\alpha$ of the $G$ action. Then the reduced Cartan model
consists of forms
\begin{equation}
u_\alpha \in\CI(Y_\alpha ;\Lambda ^*\otimes C_\alpha )
\label{EKR.138}\end{equation}
over each $Y_\alpha$ with values in $C_\alpha$ but connected by (reduced)
pull-back maps as in \eqref{EKR.179}, 
\begin{equation}
u_\alpha \big|_{H_{\beta ,\alpha }}=\phi_{\beta ,\alpha }^*u_\beta .
\label{EKR.139}\end{equation}

Again by repeated application if the reduction results above we arrive at
\begin{proposition}\label{EKR.140} The equivariant cohomology of $X$ is
naturally isomorphic to the reduced Cartan cohomology which is the deRham
cohomology of forms \eqref{EKR.138}, \eqref{EKR.139}.
\end{proposition}

Finally we \emph{define} delocalized equivariant cohomology by replacing
the Cartan space $C(H_\alpha)$ by the real representation ring of
$H_\alpha,$ $\cR(H_\alpha )=\cI(H_\alpha )\otimes \bbR,$ to which it is
isomorphic in the connected case. The action of $N(H_\alpha)$ by
conjugation on $H_\alpha$ induces an equivariant action on $\cI(H_\alpha)$
and hence on $\cR(H_\alpha)$ and this defines a flat bundle
$R_\alpha$ over $Y_\alpha.$ Then the admissible `delocalized' forms are
\begin{equation}
v_\alpha \in\CI(Y_\alpha ;\Lambda ^*\otimes \cR_\alpha )
\label{EKR.141}\end{equation}
connected by the pull-back maps augmented by branching 
\begin{equation}
v_\alpha \big|_{H_{\beta ,\alpha }}=\phi_{\beta ,\alpha }^*v_\beta.
\label{EKR.142}\end{equation}
The delocalized cohomology $H^*_{G,\dl}(X)$ is defined to be the deRham
cohomology of this complex of forms. In this sense there is no reduction
theorem but a definition. It is very natural to expect the possibility a
sheaf-theoretic formulation directly over the original space.

The surjective character maps $\cR(H)\longrightarrow C(H)$ for the isotropy groups
induce surjective `localization' maps from the forms \eqref{EKR.141},
\eqref{EKR.142} to the Cartan forms \eqref{EKR.138}, \eqref{EKR.139}
through relaxation of the local coefficients from $\cR_\alpha$
to $C_\alpha.$ This therefore descends to a surjective map
\begin{equation}
\loc:H^*_{G,\dl}(X)\longrightarrow H^*_{G}(X).
\label{EKR.145}\end{equation}

To define the equivariant Chern character as a homomorphism giving a
commutative triangle
\begin{equation}
\xymatrix{
K_G(X)\ar[rr]^-{\Ch_{G,\dl}}\ar[dr]_-{\Ch_G}&& H^{\ev}_{G,\dl}(X)\ar[dl]^-{\loc}\\
&H^{\ev}_{G}(X)
}
\label{EKR.146}\end{equation}
we start with an appropriate notion of `reduced connection' on the bundles
\eqref{EKR.135}, \eqref{EKR.136}.

Each bundle $W_\alpha$ can be equipped with a connection which is invariant
under the action of $\hN_\alpha.$ The Chern character is then an even closed
form on $X_\alpha$ which is invariant under the action of $\hN_\alpha,$
since the center acts trivially on forms thus is a true form on the base
$Y_\alpha.$ The same is true for the $\vN_\alpha$ action on $M'$ as a
bundle over $X_\alpha$ and the sum over the fibres of the product of these
closed forms is an element of the space \eqref{EKR.141}. The collection of
these forms satisfies \eqref{EKR.142} provided the connections are chosen
consistently. This defines the Chern character \eqref{EKR.146} with
commutativity following from the lifting of the standard definition of the
equivariant Chern character, see for example \cite{Berline-Getzler-Vergne1}.

\section{Atiyah-Hirzebruch isomorphism}\label{AHI}

\begin{theorem}\label{EKR.143} For any smooth action by a compact Lie group
  on a compact manifold the delocalized equivariant Chern character defines
  an isomorphism
\begin{equation}
\Ch_{G,\dl}:K_G(X)\otimes\bbR\longrightarrow H^*_{G,\dl}(X).
\label{EKR.144}\end{equation}
\end{theorem}

The proof follows the same lines as in \cite{EqKtRe}. The delocalized
equivariant cohomology is naturally $\bbZ_2$-graded (not $\bbZ$-graded
because of the Chern character factors in the augmented pull-back
maps). Odd K-theory is defined in the usual way by suspension, i.e.\ as the
null space of the restriction morphism 
\begin{equation}
K_{G,\red}(\bbS\times Y_*)\longrightarrow K_{G,\red}(\{0\}\times Y_*).
\label{EKR.147}\end{equation}

Consider the `pruning' of the isotropy tree. Thus let $A$ be an ordered
collection of the indices, so 
\begin{equation}
\beta <\alpha,\ \beta \in P\Longrightarrow \alpha \in P.
\label{EKR.148}\end{equation}
In each cohomology theory the notion of triviality on the $X_\alpha$ for
$\alpha \in P$ is well-defined giving relative versions 
\begin{equation}
K^*_{G,\red}(Y_*;P),\ H^*_{G,\dl}(Y_*;P)\Mand H^*_{G}(Y_*;P).
\label{EKR.149}\end{equation}

In all three cohomology theories there is an excision sequence
corresponding to two ordered sets $P$ and $P'=P\cup\{\gamma \}.$ Namely for
(reduced) equivariant K-theory 
\begin{equation}
\xymatrix{
K^0_{G,\red}(Y_*;P')\ar[r]&K^0_{G,\red}(Y_*;P)\ar[r]&K^0_{G,\red}(Y_\gamma ;P)\ar[d]\\
K^1_{G,\red}(Y_\gamma ;P)\ar[u]&K^1_{G,\red}(Y_*;P)\ar[l]&K^1_{G,\red}(Y_*;P').\ar[l]
}
\label{EKR.150}\end{equation}

The delocalized Chern character gives an exact complex 
\begin{equation}
\xymatrix@!=3.5pc{
K^0_{G,\red}(Y_*;P')\ar[rr]\ar[dr]^{\Ch}&&
K^0_{G,\red}(Y_*;P)\ar[rr]\ar[d]^{\Ch}&&
K^0_{G,\red}(Y_\gamma ;P)\ar[dl]^{\Ch}\ar[ddd]
\\
&H^0_{\dl}(Y_*;P')\ar[r]&H^0_{\dl}(Y_*;P)\ar[r]&H^0_{\dl}(Y_\gamma ;P)\ar[d]\\
&H^1_{\dl}(Y_\gamma ;P)\ar[u]&H^1_{\dl}(Y_*;P)\ar[l]&H^1_{\dl}(Y_*;P').\ar[l]
\\
K^1_{G,\red}(Y_\gamma ;P)\ar[uuu]\ar[ur]_{\Ch}&&
K^1_{G,\red}(Y_*;P)\ar[ll]\ar[u]^{\Ch}&&
K^1_{G,\red}(Y_*;P')\ar[ul]^{\Ch}.\ar[ll]
}
\label{EKR.151}\end{equation}

This complex remains exact when tensored with $\bbR.$ We proceed by
induction over $P$ starting from the case that $P$ contains all elements
except the minimal one corresponding to the open isotropy type. The top
right and lower left Chern character maps are isomorphism as discussed
above and by induction the other two corner Chern maps are
isomorphisms. The fives lemma shows the central maps to be isomorphism, so
the induction continues. The case that $P=\emptyset$ is the
Atiyah-Hirzebruch-Baum-Brylinski-MacPherson isomorphism.

\providecommand{\bysame}{\leavevmode\hbox to3em{\hrulefill}\thinspace}
\providecommand{\MR}{\relax\ifhmode\unskip\space\fi MR }
\providecommand{\MRhref}[2]{%
  \href{http://www.ams.org/mathscinet-getitem?mr=#1}{#2}
}
\providecommand{\href}[2]{#2}

\end{document}